\documentclass[12pt]{amsart}

\usepackage{amsmath, amssymb, amscd, verbatim, xspace,amsthm}

\usepackage{latexsym, epsfig, color}

\usepackage[OT2,T1]{fontenc}

\DeclareSymbolFont{cyrletters}{OT2}{wncyr}{m}{n}

\DeclareMathSymbol{\Sha}{\mathalpha}{cyrletters}{"58}

\newcommand{\Z}{\ensuremath{{\mathbb{Z}}}\xspace}

\newcommand{\Q}{\ensuremath{{\mathbb{Q}}}}

\newcommand\bq{\begin{equation}}

\newcommand\eq{\end{equation}}

\newtheorem{proposition}{Proposition}[section]

\newtheorem{theorem}[proposition]{Theorem}

\newtheorem{corollary}[proposition]{Corollary}

\newtheorem{example}[proposition]{Example}

\newtheorem{conjecture}[proposition]{Conjecture}

\theoremstyle{remark}

\usepackage[hidelinks]{hyperref}
\usepackage{pst-node}
\usepackage{tikz-cd} 

\usepackage[all]{xy}
\usepackage{fullpage, pdfsync}

\newtheorem{nts}{Note to self}

\title{The $2$-Class Tower of $\mathbb{Q}(\sqrt{-5460})$}

\author{Nigel Boston}
\address{Department of Mathematics\\
University of Wisconsin-Madison \\ 480 Lincoln Drive \\
Madison, WI 53706 USA}
\email{boston@math.wisc.edu}
\author{Jiuya Wang}
\address{Department of Mathematics\\
University of Wisconsin-Madison \\ 480 Lincoln Drive \\
Madison, WI 53706 USA}
\email{jiuyawang@math.wisc.edu}

\begin{document}

\begin{abstract}
The seminal papers in the field of root-discriminant bounds are those of Odlyzko and Martinet. Both papers include the question of whether the field
$\mathbb{Q}(\sqrt{-5460})$ has finite or infinite $2$-class tower. This is a critical case that will either substantially lower the best known upper bound for lim inf of root-discriminants (if infinite) or else give a counter-example to what is often termed Martinet's conjecture or question (if finite). Using extensive computation and introducing some new techniques, we give strong evidence that the tower is in fact finite, establishing other properties of its Galois group en route.
\end{abstract}

\maketitle

\section{Introduction}
If $K$ is a number field, then its root discriminant $rd(K)$ is an important invariant defined to be the $[K:\Q]$th root of the absolute value of its discriminant. In this paper we are concerned with the multi-set of real numbers $rd(K)$ as $K$ varies. Its smallest values are easily computed to be  $1, \sqrt{3}, 2, \sqrt{5}, ...$ (with multiplicity one). For large enough $C$ there are infinitely many number fields with $rd(K) \le C$. The main question is to find the smallest such $C$, in other words $C_0:=$ lim inf $rd(K)$. It has long been known [18] that $C_0 \geq 4\pi e^\gamma$ and that under GRH $C_0 \geq 8\pi e^\gamma \approx 44.76$. As for upper bounds, the best obtained so far [10] gives $C_0 < 82.2$. 

Upper bounds are typically obtained as follows. If $L/K$ is an unramified extension, then $rd(L)=rd(K)$. It follows that if $K$ has infinitely many unramified extensions, or equivalently if the Galois group of its maximal unramified extension is infinite, then $C_0 \leq rd(K)$. In practice it is hard to compute this Galois group for a given $K$. We can usually say much more about the Galois group of the maximal unramified $p$-extension of $K$, meaning the compositum of all unramified Galois extensions of $p$-power degree, for some prime $p$. This extension is also called the \it{$p$-class tower}\rm \ of $K$ and its Galois group will be termed the \it{$p$-tower group}\rm. A few recent papers [5], [6], [17], have studied the $2$-tower groups of several quadratic fields with root discriminant less than $82.2$ in attempts to improve upon the upper bound for $C_0$, but in each case so far the group has turned out to be finite. One of the few unresolved instances remaining is that of $\Q(\sqrt{-5460})$, a case that was mentioned as important all the way back in Odlyzko's seminal paper [19]. It is the smallest imaginary quadratic field whose $2$-class group has rank $4$, which implies that the $2$-tower group is a pro-$2$ group with $4$ generators. This considerably increases the complexity of the problem of determining whether this group is finite or infinite, and the current paper is the first serious attempt to address these complications. If its $2$-class tower turns out to be infinite, then the known upper bound for $C_0$ will be lowered to $\sqrt{5460} \approx 73.89$.

In [15], Martinet noted that imaginary quadratic fields with $2$-class group of rank $>4$ have infinite $2$-class tower by Golod-Shafarevich [9]. He raised the question of whether those of rank equal to $4$ always have infinite $2$-class tower. This has inspired a lot of recent work. Benjamin [1],[2],[3],  Sueyoshi [23],[24],[25], Mouhib [16], and Wang [26] have established that the tower is infinite in many cases, most of which have nonzero $4$-rank. The case of zero $4$-rank, which includes that of $\Q(\sqrt{-5460})$, is the most challenging, and the case at hand is as yet still open.

In fact, below, we give evidence that the $2$-class tower of $K = \Q(\sqrt{-5460})$ is finite (but very large). This runs contrary to what is usually conjectured [15],[26]. We show that its Galois group $G_K$ has $4$ generators and exactly $5$ relators and compute information on its maximal quotient $Q_c(G_K)$ of nilpotency class $c = 2, 3$. In [15], Martinet noted that he knew of no finite $2$-group with $4$ generators and $5$ relators. Our method yields millions of such groups. It is seen that almost all of our candidates for $G_K$ are finite and that often, for a given $c$ and candidate for $Q_c(G_K)$, all corresponding candidates for $G_K$ are finite. 

The key empirical observation is that for each candidate for $Q_3(G_K)$ all its descendants $G$ have the property that $G^8$ is abelian and of finite, bounded index. Once checked in its entirety, this will imply that $G_K$ is finite.

We make extensive use of the computer algebra systems Magma [4] and Pari Gp.

\section{$2$-Tower Groups of Imaginary Quadratic Fields}
Let $K$ be an imaginary quadratic field and let $G$ be its $2$-tower group. Then $G$ is a (topologically) finitely presented pro-$2$ group. Let $d(G)$ and $r(G)$ denote its generator and relator rank respectively. By Burnside's basis theorem, $d(G) = d(G^{ab})$, where $G^{ab}$ is the maximal abelian quotient of $G$, which here is isomorphic to the $2$-class group of $K$. As for the relator rank, by [21], $r(G)$ is known to equal $d(G)$ or $d(G)+1$. 

Moreover, by Golod-Shafarevich [9], if $G$ is finite, then $r(G) > d(G)^2/4$. Combining the last two statements gives that if $d(G) \geq 5$, then $G$ is infinite.

Let $K$ denote $\Q(\sqrt{-5460})$ and $G_K$ its $2$-tower group. This paper will provide lengthy information regarding the structure of $G_K$. The first such information is that $G_K^{ab} \cong (\Z/2)^4$, which we denote $[2,2,2,2]$. (In general, the group $\Z/2^{k_1} \times ... \times \Z/2^{k_d}$ will be denoted $[2^{k_1},...,2^{k_d}]$ for short). It follows that $d(G_K) = 4$ and so $r(G_K) = 4$ or $5$. There are known to exist finite $2$-groups with $d(G)=4$ and $r(G)=5$ (the smallest [11] has order $2^{14}$ - our method below yields millions of $4$-generator $5$-relator $2$-groups). 

\section{Ingredients regarding $p$-Group Theory}
If $G$ is a finite $p$-group, then its \it{$p$-central series of subgroups}\rm \ is defined by $P_0(G)=G, P_{n+1}(G)=[G,P_n(G)]P_n(G)^p (n \geq 0)$. If $G$ is a finitely generated pro-$p$ group, then we take the closed subgroups generated by these expressions. In particular, $P_1(G)$ is the Frattini subgroup of $G$. If $P_c(G)=\{1\}$, but $P_{c-1}(G) \ne \{1\}$, then we say that $G$ has $p$-class $c$. The unique (up to isomorphism) $p$-group with $d$ generators and $p$-class $1$ is the elementary abelian $p$-group $(\Z/p)^d$. The quotient $Q_c(G):= G/P_c(G)$ is the maximal quotient of $G$ of $p$-class $\leq c$ and is finite. If $G$ has abelianization of exponent $p$ (as is the case for the Galois group of interest to us), then in fact all the factors in its lower central series have exponent $p$, and so its $p$-central series coincides with its lower central series and its $p$-class is the same as its nilpotency class [20, section 5.2.5].

Given $p$ and $d$, the set of (isomorphism classes of) finite $d$-generator $p$-groups forms a rooted tree with root $(\Z/p)^d$. An edge joins $G$ to $H$ if $G$ has $p$-class $c$ and $H \cong Q_{c-1}(G)$. Thus, the vertices at length $k$ from the root are precisely the groups of $p$-class $k+1$. O'Brien has provided an algorithm [18] that, given a $p$-group $G$ of $p$-class $c$, yields all its neighbors of $p$-class $c+1$. These are called the \it{immediate descendants}\rm \ (or \it{children}\rm) of $G$. Some groups (such as the quaternion group of order $8$) have no descendants and are called \it{terminal}\rm. An equivalent characterization is that their nuclear rank is $0$. The nuclear rank (defined below) of a group is loosely correlated with how many children it has.

Let $G$ be a finite $d$-generator $p$-group of $p$-class $c$. It can be presented as $F/R$, where $F$ is a free group on $d$ generators. Let $R^\ast = [R,F]R^p$, a characteristic subgroup of $R$. The group $G^\ast:= F/R^\ast$ is a finite group called a $p$-covering group of $G$ and the quotient $R/R^\ast$ is the $p$-multiplicator, a finite elementary abelian group. The \it{$p$-multiplicator rank}\rm \ of $G$ is the rank of $R/R^\ast$. The nucleus of $G$ is $P_c(G^\ast) = P_c(F^\ast)R^\ast/R^\ast \leq R/R^\ast$, and its rank defines the \it{nuclear rank}\rm \ of $G$. These are all important players in O'Brien's algorithm. In particular, every child of $G$ arises as a quotient of $G^\ast$. His algorithm sorts out isomorphism classes of such quotients 
$F/M$, where $M/R^\ast$ is a proper subgroup of the $p$-multiplicator that supplements the nucleus.

In [5], we note the following useful proposition:

\begin{proposition}
If $G$ is a finite $p$-group, then for any $c$
$$ \text{p-multiplicator rank}(Q_c(G)) - \text{nuclear rank}(Q_c(G)) \leq r(G) $$
\end{proposition}

If $G$ is an infinite $d$-generator pro-$p$ group, then it is determined by its quotients $Q_c(G)$ ($c=1,2,...$), which define an infinite path through the above tree. We can therefore locate $G$ as an end of this tree.
The $2$-tower group $G_K$ of $K:= \Q(\sqrt{-5460})$ is somewhere within the O'Brien tree (if finite) or at an end (if infinite). We will successively search for $Q_c(G_K)$ ($c=1,2,...$) and hence $G_K$ within this tree. Our main tool is to prune the tree using information about finite unramified field extensions $L$ of $\Q(\sqrt{-5460})$. Such a field $L$ is the fixed field of an open subgroup $H$ of $G_K$ and the $2$-class group of $L$ tells us $H^{ab}$. 

\section{Beginning the Search}
Let $G_K$ be the $2$-tower group of $\Q(\sqrt{-5460})$. As noted above, $Q_1(G_K) \cong (\Z/2)^4$. We next determine $Q_2(G_K)$ and simultaneously show that $r(G_K)=5$. We know that $Q_2(G_K)$ is one of the children of $(\Z/2)^4$. Call it $Q$. There are $15170$ such children. First, note that $Q^{ab} \cong [2,2,2,2]$. Only $7851$ of the children satisfy this. Applying proposition 3.1, using $r(G_K) \leq 5$ reduces us to $4722$ possibilities for $Q$.

Next, we can use the abelianizations of low index subgroups obtained by computing $2$-class groups of unramified extensions of small degree, as indicated at the end of the last section.

\begin{proposition}
$G_K$ has $15$ subgroups of index $2$. Their abelianizations are $[2,4,4]$ ($8$ times), $[2,4,8]$ ($2$ times), $[4,4,4]$ ($1$ time), $[2,2,2,4]$ ($2$ times), and $[2,2,4,4]$ ($2$ times). 
\end{proposition}

This still does not cut us down to one possibility for $Q_2(G_K)$, but work of Koch from 1964 [12], building on that of Fr\"ohlich [8], does give an explicit presentation for $Q_2(G_K)$.

\begin{theorem}
The Galois group of the maximal unramified $2$-extension of nilpotency class $2$ of $\Q(\sqrt{-5460})$, $Q_2(G_K)$, is the $2$-class $2$ quotient of the pro-$2$ group of order $2^9$ and presentation
$$\langle a,b,c,d \mid a^{-2}(d,c),b^{-2}(d,a)((d,b),b),c^{-2}(b,a)((d,b),b),d^{-2}(c,a)(d,a)(d,b),(b,c) \rangle$$
\end{theorem}

This group has $p$-multiplicator rank $11$ and nuclear rank $6$. Proposition 3.1, together with [21], then implies the following.

\begin{corollary}
The relator rank of the $2$-tower group of $\Q(\sqrt{-5460})$ is $5$.
\end{corollary}

The large nuclear rank suggests that $Q_2(G_K)$ has many children, and indeed it has $151501$ of them. Of these $32768$ satisfy Propositions 3.1 and 4.1. This is still a huge number to handle and unfortunately there is no extension of Theorem 4.2 to higher nilpotency class and so we introduce some new techniques.

\section{Abelianizations of Larger Index Subgroups}
We can also compute the $2$-class group for every unramified degree $4$ extension of $K$, which then tells us the lattice of subgroups of $G_K$ of index at most $4$, together with their abelianizations.

\begin{proposition}
Up to conjugacy, $G_K$ has $51$ subgroups of index $4$. Their abelianizations are $[2,4,4]$ ($2$ times), $[2,4,8]$ ($8$ times), $[4,4,8]$ ($3$ times), $[2,2,2,4]$ ($2$ times), $[2,2,4,4]$ ($14$ times), $[2,2,4,8]$ ($9$ times), $[2,2,4,16]$ ($5$ times), $[2,2,8,8]$ ($1$ time), $[2,4,4,4]$ ($3$ times), $[2,4,4,8]$ ($1$ time), $[2,2,2,2,4]$ ($1$ time), $[2,2,2,4,4]$ ($1$ time), $[2,2,2,8,8]$ ($1$ time). 
\end{proposition}

We also recorded lattice data, i.e. which index $2$ subgroups contain which index $4$ subgroups. Note that there are $5$ conjugacy classes of subgroups that have a unique abelianization. We call these \it{critical}\rm \ subgroups of index $4$. We computed the abelianizations of their maximal subgroups (using pari gp to compute class groups of degree $16$ extensions of $\Q$). This yielded the following.

\begin{proposition}
As regards the $5$ critical subgroups of index $4$, \\
1) the unique one with abelianization $[2,2,8,8]$ has index $2$ subgroups with abelianizations $[2,2,8,16]$ ($10$ times), $[4,4,8,16]$ ($2$ times), $[2,2,4,4,16]$ ($1$ time), $[2,2,2,8,16]$ ($1$ time), $[2,2,2,2,8,16]$ ($1$ time); \\
2) the unique one with abelianization $[2,4,4,8]$ has index $2$ subgroups with abelianizations $[4,8,8]$ ($4$ times), $[2,2,8,8]$ ($4$ times), $[2,4,4,8]$ ($4$ times), $[2,2,4,4,8]$ ($1$ time), $[2,2,4,4,16]$ ($1$ time),$[2,2,2,4,8,8]$ ($1$ time);\\
3) the unique one with abelianization $[2,2,2,2,4]$ has index $2$ subgroups with abelianizations $[2,2,2,8]$ ($4$ times), $[2,2,4,4]$ ($6$ times), $[2,2,4,8]$ ($16$ times), $[2,2,2,2,4]$ ($2$ times), $[2,2,2,4,8]$ ($1$ time), $[2,2,2,8,8]$ ($1$ time), $[2,2,4,4,8]$ ($1$ time);\\
4) the unique one with abelianization $[2,2,2,4,4]$ has index $2$ subgroups with abelianizations $[2,2,8,8]$ ($8$ times), $[2,4,4,8]$ ($6$ times), $[4,4,4,4]$ ($2$ times), $[2,2,2,4,8]$ ($8$ times), $[2,2,4,4,4]$ ($2$ times), $[2,2,4,4,16]$ ($1$ time), $[2,2,2,2,4,4]$ ($1$ time), $[2,2,2,2,8,8]$ ($2$ times), $[2,2,2,4,8,8]$ ($1$ time); \\
5) the unique one with abelianization $[2,2,2,8,8]$ has index $2$ subgroups with abelianizations $[2,2,8,8]$ ($4$ times), $[2,2,8,16]$ ($6$ times), $[2,4,8,8]$ ($8$ times), $[2,4,8,16]$ ($4$ times), $[2,2,2,8,16]$ ($4$ times), $[2,2,4,4,8]$ ($1$ time), $[2,2,4,8,8]$ ($1$ time), $[2,2,4,16,16]$ ($2$ times), $[2,2,2,4,8,8]$ ($1$ time); \\	
\end{proposition}




\section{Capitulation}

The index $4$ abelianization information from Section 5 cuts the possibilities down from $32768$ to $4096$ candidates for $Q_3(G_K)$. Another technique we can usefully employ is capitulation. If $H$ is a finite index subgroup of $G_K$, this concerns the transfer map $\varphi: G_K^{ab} \rightarrow H^{ab}$. 

Let $H$ be a finite index subgroup of $G_K$ corresponding to extension $L$ of $K$. Then $H$ is the $2$-tower group of $L$. As noted earlier, Class Field Theory gives an isomorphism between $G_K^{ab}$ and the $2$-class group of $K$ and between $H^{ab}$ and the $2$-class group of $L$. Consider the following commutative square, where the upper horizontal map is the transfer map and the lower horizontal map is given by sending integral ideal $\alpha \subset O_K$ to the ideal class of $\alpha O_L$:

\[\begin{tikzcd}
  G_K^{ab} \arrow[r, "\varphi"] \arrow[d, "iso"]
    & H^{ab} \arrow[d, "iso"] \\
  Cl_2(K) \arrow[r, "\alpha O_L"]
    &Cl_2(L)
\end{tikzcd}\]

The vertical maps are all isomorphisms, and so the kernel is the same for both the upper and lower horizontal maps. The strategy is to compute the kernel of the lower map using Class Field Theory and compare it with a computation using group theory of the kernel of the upper map. 

On the group side, however, we only have some candidate quotient $Q$ of $G_K$ of $2$-class $3$ to work with. Let $\bar{H}$ denote the image of $H$ in $Q$. Consider the commutative square:

\[\begin{tikzcd}
Q^{ab} \arrow[r, "\phi"]
& \bar{H}^{ab}  \\
G_K^{ab} \arrow[r, "\varphi"]  \arrow[u, two heads]
& H^{ab}\arrow[u, two heads]\\
\end{tikzcd}\]

In general, it is difficult to determine the kernel of $\phi$, but in our situation all $4096$ children have the same abelianization as $G_K^{ab}$, and the abelianization of each index $2$ subgroup matches the correct field data, and so, if $H$ has index $2$ in $G_K$, both the vertical maps are isomorphisms. We save those children where the group-theoretical kernel matches the true (number-theoretical) data. In fact, there could only be different kernels of $\phi$ for the $13$th and $14$th subgroups (as listed by Magma) of index $2$. Saving the cases with a match reduces us to $512$ possibilities for $Q_3(G_K)$. These all have order $2^{15}$ and nuclear rank $5$ or $6$ (and so have many children). 



\section{Nilpotency class $3$ quotient of $G_K$}

So far we have $512$ candidates for $Q_3(G_K)$. We will now explain how extensive calculations reduce us to just $2$ possibilities (the $14$th and $52$nd candidates in the list produced by Magma as described above). The strategy is to use the abelianizations of index $4$ subgroups of $G_K$ and of index $8$ subgroups contained inside critical subgroups of index $4$, i.e. Propositions 5.1 and 5.2. The corresponding abelianizations of the $512$ candidates for $Q_3(G_K)$ are all consistent with this data, but for some of them all their children or grandchildren fail. For example, $384$ of them (the last $384$ in Magma's list) can be eliminated by using the following criterion of Nover [5].

\begin{proposition}
	Let $G$ be a pro-$p$ group, $N$ a finite index normal subgroup, and $V$ a word. Assume $P_c(G) \le N$, if $[G/P_c(G): V(N/P_c(G))] = [G/P_{c+1}(G): V(N/P_{c+1}(G))]$, then $(N/P_c(N) )/V(N/P_c(G)) \simeq N/V(N)$. 
\end{proposition}

We apply this with $N$ the $5$th critical subgroup. For those $384$ groups, it has abelianization $[2,2,2,4,8]$. It turns out that for all their children too, this subgroup has abelianization $[2,2,2,4,8]$. Applying Nover's criterion, this subgroup will have abelianization $[2,2,2,4,8]$ for every descendant and so will never reach its true value of $[2,2,2,8,8]$. This eliminates those $384$ groups.

By further use of low index abelianizations, we conjecturally get down to the following two possibilities for $Q_3(G_K)$. The point is that whereas each of our $128$ groups typically has tens of thousands of children, there are much more limited possibilities for the abelianizations of the index $\leq 4$ subgroups and of the index $8$ subgroups contained inside the critical subgroups. In fact, each of the $128$ candidate groups has either $2, 4,$ or $8$ such possible lists of abelianizations. For a given group $G$, this set of lists can be found in practice by using a Magma program that randomly takes quotients of the $p$-covering group of $G$ to produce children of $G$. Since each list arises equally often, we expect only to need at most about $8\ \text{ln}(8) \approx 17$ such quotients.

Likewise, we can produce such lists for further descendants of the $128$ groups and compare them with the true list as given in Propositions 5.1 and 5.2. We find that matches apparently only arise for descendants (in fact grandchildren) of $2$ of the $128$ candidates. If Nover's lemma were true for non-normal subgroups $N$, then this would yield a proof - as things stand, we have plenty of evidence for the following conjecture.

\begin{conjecture}
The Galois group of the maximal unramified $2$-class $3$ $2$-extension of $\Q(\sqrt{-5460})$, $Q_3(G_K)$, is the $3$-class $2$ quotient of the pro-$2$ group with presentation 

\begin{equation}
\begin{aligned}
\langle a,b,c,d &\mid a^{-2}(d,c),b^{-2}(d,a)((d,b),b),c^{-2}(b,a)((d,b),b),\\
&d^{-2}(c,a)(d,a)(d,b)(b,a,a)(c,a,a)(d,a,a),(b,c) \rangle
\end{aligned}
\end{equation}
or
\begin{equation}
\begin{aligned}
\langle a,b,c,d &\mid a^{-2}(d,c),b^{-2}(d,a)((d,b),b),c^{-2}(b,a)(b,a,d)(c,a,a)((d,b),b), \\
&d^{-2}(c,a)(d,a)(d,b)(b,a,a)(b,a,d),(b,c) \rangle .
\end{aligned}
\end{equation}

Each has order $2^{15}$ and nuclear rank $6$.
\end{conjecture}

As for understanding quotients of higher nilpotency class, we need to introduce more tools in the next two sections.

\section{Moribund $p$-groups}

We call a $p$-group \it{moribund}\rm \ if all its descendants are finite, in other words if there is no infinite path in the O'Brien tree starting at the given group. If our candidates for $Q_3(G_K)$ were moribund, then we would conclude that $G_K$ must be finite. This is not true for them, but holds for many of their descendants. There is a simple test for a $p$-group to be moribund.

\begin{theorem} 
Suppose that $G$ is a finite $p$-group. Let $G_1$ be its $p$-covering group, $G_2$ the $p$-covering group of $G_1$, and so on. If there exists $c$ such that the $G_c$ has nuclear rank $0$, then $G$ is moribund.
\end{theorem}

This test is used to avoid the need to compute huge numbers of descendants and to show that if $G_K$ is infinite, then it must lie in certain areas of the O'Brien tree. For instance, there are $224$ moribund groups among the original $512$ candidates for $Q_3(G_K)$. 







\section{Nilpotency class $>3$ quotients of $G_K$}

We have made extensive investigation of the descendants of the $128$ candidates for $Q_3(G_K)$ and in particular of the two groups in Conjecture 7.2. For those two candidates we restrict to descendants satisfying Propositions 5.1 and 5.2, which is a moot restriction for large enough $2$-class. We have found millions of such groups, with order as high as $2^{81}$ and nilpotency class as high as $14$. 

The most convincing evidence that $G_K$ is finite comes from the following observation. If $G$ is a $2$-group, we define $G^n$ to be the subgroup of $G$ generated by all $n$th powers. If $G$ is a pro-$2$ group, then $G^n$ denotes the closed subgroup generated by all $n$th powers.

\begin{conjecture}
Let $G$ be any  descendant of one of the $128$ candidates for $Q_3(G_K)$. Then $G^8$ is of index at most $2^{40}$ and is abelian.
\end{conjecture}

For instance, the group $G$ of order $2^{81}$, mentioned above, has $G^8$ of index $2^{35}$ and abelian of rank $33$. Nover's lemma can be used to show that if the index of the subgroup generated by $8$th powers stabilizes when moving from a group to a child of that group, then the index is the same for all further descendants. The main consequence of the above conjecture is the following.

\begin{corollary}
Conjecture 9.1 implies that $G_K$ is finite.
\end{corollary}

\begin{proof}
Suppose $G_K$ is infinite. Then by Conjecture 9.1, for the sequence of its $2$-class quotients, the index of the subgroup generated by $8$th powers has to stabilize and the subgroup is abelian. In that case $G_K$ has an infinite abelian subgroup of finite index. The corresponding number field has infinite $2$-class group, which is impossible.
\end{proof}

In fact the $128$ candidates for $Q_3(G_K)$ fall into two classes, the first of which has considerably simpler descendants (for instance, this includes all moribund cases) and for any of these it appears that $G^4$ is abelian of index at most $2^{11}$. This is not true for all descendants of the $2$ candidates in Conjecture 7.2, both of which fall into the second class. It is, however, enlightening to investigate descendants of groups in the first class. For instance, we found the following infinite descendant of the 1st of the $128$ candidates.

\begin{example}
The $1$st of the $128$ groups (as listed by Magma) has an infinite descendant. This pro-$2$ group has presentation
\begin{equation}
\begin{aligned}
\langle a,b,c,d \mid & a^{-2}(d,c),b^{-2}(d,a)(d,b,b)(b,a,a,c),\\
& c^{-2}(b,a)(d,b,b)(b,a,a,c),d^{-2}(c,a)(d,a)(d,b),(b,c) \rangle
\end{aligned}
\end{equation}
\end{example}

For this group $G$, the subgroup $G^4$ is abelian of index $2^{11}$. Taking the center of the centralizer of $G^4$ yields an abelian normal subgroup $A$ of index $2^8$ with $A \cong [2,4,4,8] \times (\Z_2)^4$. It follows that $G$ has a free abelian normal subgroup of finite index and rank $4$. In fact $G$ is a $2$-adic pre-space group, as studied in the solution of the coclass conjecture [13],[22].

\vskip.25in

\centerline{\bf{Acknowledgements}}\rm

The authors thank Charles Leedham-Green for his comments on the paper. The first author was supported by Simons Foundation Award MSN-179747. The second author was supported by National Science Foundation grant DMS-1301690.

\newpage

\centerline{\bf{References}}\rm
\vskip.25in

1. E.Benjamin, On imaginary quadratic number fields with $2$-class group of rank $4$ and infinite $2$-class field tower, Pacific J. Math. 201 (2001), 257-266.

2. E.Benjamin, On a question of Martinet concerning the $2$-class field tower of imaginary quadratic number fields, Ann. Sci. Math. Quebec 26 (1) (2002), 1-13.

3. E.Benjamin, On the $2$-class field tower conjecture for imaginary quadratic number fields with $2$-class group of rank $4$, J. Number Theory 154 (2015), 118-143.

4. W.Bosma, J.J.Cannon, and C.Playoust, The Magma algebra system. I. The user language, J.Symbolic Comput. 24 (1997), 235-265.

5. N.Boston and H.Nover, Computing pro-$p$ Galois groups, Lecture Notes in Computer Science 4076 (2006), ANTS VII, 1-10.

6. M.R.Bush, Computation of the Galois groups associated to the 2-class towers of some quadratic fields, J. Number Theory 100 (2003), 313-325.

7. J.-M.Fontaine and B.C.Mazur, Geometric Galois representations, Elliptic curves, modular forms, \& Fermat's last theorem (Hong Kong, 1993), Series in Number Theory, 1, Int. Press, Cambridge, MA, 41-78.

8. A.Fr\"ohlich, Central extensions, Galois groups, and ideal class groups of number fields, Contemporary Mathematics 24, AMS (1983).

9. E.S.Golod and I.R.Shafarevich, On the class field tower (Russian), Izv. Akad. Nauk SSSR Ser. Mat. 28 (1964), 261-272; English translation in AMS Transl. 48 (1965), 91-102.

10. F.Hajir and C.Maire, Tamely ramified towers and discriminant bounds for number fields. II., J. Symbolic Comput. 33 (2002), 415-423.

11. G.Havas, M.F.Newman, and E.A.O'Brien, On the efficiency of some finite groups, Comm. Algebra 32 (2004), 649-656.

12. H.Koch, Galois theory of $p$-extensions, Springer Monographs in Mathematics. Springer-Verlag, Berlin (2002).

13. C.R.Leedham-Green, The structure of finite $p$-groups, J. London Math. Soc. 50 (1) (1994), 49-67.

14. A.Lubotzky and A.Mann, Powerful $p$-groups. II. $p$-adic analytic groups, J. Algebra 105 (2) (1987), 506-515.

15. J.Martinet, Tours de corps de classes et estimations de discriminants (French), Invent Math 44 (1978), 65-73.

16. A.Mouhib, Infinite Hilbert $2$-class field tower of quadratic number fields, Acta Arith. 145 (3) (2010), 267-272.

17. H.Nover, Computation of Galois groups associated to the $2$-class towers of some imaginary quadratic fields with $2$-class group $C_2 \times C_2 \times C_2$, J. Number Theory 129 (1) (2009), 231-245.

18. E.A.O'Brien, The $p$-group generation algorithm, J. Symbolic Comput. 9 (1990), 677-698.

19. A.Odlyzko, Lower bounds for discriminants of number fields, Acta Arith.  29 (1976), 275-297.

20. D.J.S.Robinson,  A course in the theory of groups, Springer-Verlag, Berlin (1982).

21. I.R.Shafarevich, Extensions with prescribed ramification points (Russian), IHES Publ. 18 (1963), 71-95; English translation in I.R.Shafarevich, Collected Mathematical Papers, Springer, Berlin (1989).

22. A.Shalev, The structure of finite p-groups: effective proof of the coclass conjectures, Invent. Math. 115 (2) (1994), 315-345.

23. Y.Sueyoshi, Infinite $2$-class field towers of some imaginary quadratic number fields, Acta Arith. 113 (3) (2004), 251-257.

24. Y.Sueyoshi, On $2$-class field towers of imaginary quadratic number fields, Far East J. Math. Sci. (FJMS) 34 (3) (2009), 329-339.

25. Y.Sueyoshi, On the infinitude of $2$-class field towers of some imaginary quadratic number fields, Far East J. Math. Sci. (FJMS) 42 (2) (2010), 175-187.

26. V.Y.Wang, On Hilbert $2$-class fields and $2$-towers of imaginary quadratic number fields, https://arxiv.org/abs/1508.06552.

\end{document}